\documentclass[10pt]{article}
\usepackage{graphicx}
\usepackage{latexsym}
\usepackage{amsmath}
\usepackage{amssymb}
\usepackage{amsthm}
\usepackage{hyperref}
\usepackage{tikz}
\def \Q {{\mathbb Q}}

\def \Z {{\mathbb Z}}

\def \C {{\mathbb C}}

\def \Ll {{\mathcal L}}

\def \J  {{\mathbf J}}

\newtheorem{theorem}{Theorem}[section]

\newtheorem{lemma}[theorem]{Lemma}

\newtheorem{obs}[theorem]{Observation}

\title{Few Non-derogatory Directed Graphs from Directed Cycles }

\author { A. Satyanarayana Reddy\\
Department of 
Mathematics, Shiv Nadar 
University, India-201314\\ (e-mail: 
 satyanarayana.reddy@snu.edu.in).
  }

\date{}
\begin{document}
\maketitle

\begin{abstract}
We constructed a few non-derogatory digraphs by adding arcs to a directed cycle and computed their
characteristic polynomials and exponents.
\end{abstract}

\section{Introduction and preliminaries}
A digraph (directed graph) $X= (V,E)$ consists of a finite set $V$, called the set of
vertices  and a
set $E$, called the set of arcs. If $(i,j)\in E$,
then $i$ and $j$ are adjacent and
$(i,j)$ is an arc starting at vertex $i$ and terminating at vertex $j$.
The adjacency matrix of a digraph $X$, denoted by $A(X)$ (or simply $A$), is the
matrix
 whose $ij^{th}$ entry $a_{i,j}$
is the number of arcs starting at $i$ and terminating at $j$.
In this work, except in few cases where
there are  $2$ or more self loops at a vertex,  we
consider  $a_{i,j}\in\{0,1\}$.  If $X^c$ is the
complement digraph of $X$,
then $A(X^c)=\J-I-A(X)$,
where $\J$ is the matrix with each entry being $1$ and $I$ is the identity
matrix.

 The characteristic polynomial of $X$ is denoted
by $\Psi_X(x)$ and it is defined as the characteristic polynomial of
the adjacency matrix $A$ of $X$, {\it{i.e.,}} $\Psi_X(x)=|xI-A|$. By
Cayley-Hamilton theorem  $\Psi_X(A)=0$. The monic polynomial $f(x)$
of least degree for which $f(A)=0$ is called the minimal polynomial
of $A$, denoted by $m_X(x)$. By definition and division algorithm in
$\C[x]$, $m_X(x)$ divides $f(x)$ for all $f(x)$ for which $f(A)=0$.
A digraph $X$ is called {\em non-derogatory} if its adjacency matrix $A$
is non-derogatory, {\it{i.e.,}} if $\Psi_X(x)=m_X(x)$; otherwise,
$X$ is called {\em derogatory}. Since $\Psi_X(x)$ and $m_X(x)$ have the
same roots, hence if all the eigenvalues of a digraph  are distinct,
then it is a non-derogatory digraph. We start with the following
theorem.

\begin{theorem}\cite{deng}
 If the adjacency matrix of a digraph $X$ of order $n$ contains a non-singular
lower
(upper) triangular sub matrix of order $n-1$, then $A(X)$ is non-derogatory.
\end{theorem}
Note that the following matrix has a non-singular upper triangular
matrix of order $n-1$. Hence corresponding digraphs are
non-derogatory.
$$ \begin{bmatrix}
    a_{1,1} & 1 &a_{1,3}  &\dots  & a_{1,n} \\
    a_{2,1} & 0 & 1 & \dots  & a_{2,n} \\
       \vdots &\vdots & \ddots & \ddots  & \vdots \\
     a_{n-1,1} & 0 & 0 & \dots  & 1 \\
a_{n,1} & a_{n,2} &a_{n,3} & \dots  & a_{n,n} \\
     \end{bmatrix}$$

Our objective is to show that  these digraphs have distinct
eigenvalues whenever $a_{n,1}=1$. As there are lots of digraphs
satisfying $a_{n,1}=1$, we  consider only few digraphs and  evaluate
their characteristic polynomial for each case. The following well
known theorem is used to get the coefficients of the characteristic
polynomial for each case. To state the theorem, recall that a {\em linear
directed graph}({\emph{ldg}}) is a digraph in which  indegree and
 outdegree of each vertex is equal to 1 {\it{i.e.,}} it consists of directed cycles.
Hence length
of ldg (number of arcs in the ldg) is equal to number of vertices in the ldg.

\begin{theorem}[\cite{C:D:S},Theorem 1.2] \label{thm:coe}Let
$\Psi_X(x)=x^n+a_1x^{n-1}+a_2x^{n-2}\dots +a_{n-1}x+a_n$ be the characteristic
polynomial of the digraph $X$.
 Then for each $i=1,2,\ldots,n$
\begin{eqnarray*}
a_i=\sum_{L\in \Ll_i}(-1)^{p(L)},\end{eqnarray*}
where $\Ll_i$ is the set of all linear directed subgraphs({\emph{ldsgs}}) of $X$
on
exactly $i$ vertices and
$p(L)$ denotes the number of components of $L$.
\end{theorem}

Note that in almost all the cases we choose
these digraphs in such a way that every ldsg of every length contains a common
vertex
so from Theorem~\ref{thm:coe}, $p(L)=1$ for all $L\in \Ll_{i}\;for\;1\leq i\leq
n$.
 Consequently, the characteristic polynomial of
these digraphs is of the form $\Psi_X(x)=x^n-a_1x^{n-1}-a_2x^{n-2}\dots
-a_{n-1}x-a_n$, where
 {\emph{$a_i$ is the number of  ldsgs $L$ of $X$ with exactly
$i$ vertices}}.
 Hence $a_i$ is also equal to number of ldsgs $L$ of $X$ of length $i$.
Further, we use two digraphs to explain Theorem~\ref{thm:coe}. The proof for the other digraphs is
similar in nature and hence is omitted.

A matrix $B$ is said to be {\emph {cogradient}} to a matrix $C$ if there exists a
permutation matrix $P$ such that $B=P^TCP$.
 A non-negative matrix (every entry is $\geq 0$) $A$ is called
{\emph{reducible}} if there exists square submatrices $Q$ and $S$  such that $A$ is cogradient to
matrix of the form
$$\left[\begin{array}{cc}
 Q & R \\
 O & S
 \end{array}\right],$$  else  $A$
is said to be  {\emph{irreducible}}.
 It is known that the adjacency matrix of a digraph is irreducible if and only
if its digraph is strongly connected.
A non-negative  matrix is said to be {\emph{primitive }} if $A^m$ is positive
for
some positive integer $m$ and the smallest positive  integer $k$ such that $A^k$ is positive is
called
the exponent of $A$, denoted  $exp(A)$.
It is clear that a primitive matrix is necessarily an irreducible matrix.
 A digraph is said to be {\em primitive} if its adjacency matrix is primitive and its
exponent is  same as that of its adjacency matrix.  For more information on
irreducible
matrices and primitive matrices, see~\cite{minc}. We also rely on the following
known result for finding the  exponents for few of these digraphs. Again the
proof for finding the exponents of all these digraphs is similar in nature and hence we
provide a proof for only one digraph.

\begin{theorem}[\cite{C:D:S},Theorem 1.6]\label{thm:walk}
Let $A$ be the adjacency matrix of the digraph $X$ with the vertex set
$\{1,2,\dots n\}$.
 If $a_{ij}^k$ denotes the $ij^{th}$ entry of the power matrix $A^k$, then $a_{ij}^k$
is the number of directed walks of length $k$
starting at vertex $i$ and terminating at vertex $j$.
\end{theorem}

The digraphs which we are studying in this paper belongs to one of the following
classes.

{\begin{itemize}
\item Let $\mathcal{C}DC_n$ be a class of  digraphs (called {\em directed cycles with
directed chords}) of order $n\geq 3$ such that each digraph in it
contains  a directed cycle $DC_n$ with vertices labeled as
$1,2,\ldots,n$ with some additional  arcs among  non-consecutive vertices (we
call these arcs as directed chords).

 \item Let $\mathcal{C}DF_n$ be a class of  digraphs (called {\em directed fan graphs
with spokes (arcs)}) of
order $n\geq 3$, such that each digraph in it
contains  a directed path $DP_{n-1}$ (with
vertices labeled as $2,3,\ldots, n$) and an additional vertex $1$. Also for each $i$ either there is
a directed arc from $1$ to $i$ or from $i$ to $1$. Hence $\mathcal{C}DF_n$ is a class with $2^{n-1}$
digraphs of order $n$. Note that  if $X\in \mathcal{C}DF_n$ and $X$ contains an arc from $n$ to $1$
and
from $1$ to $2$, then $X\in \mathcal{C}DC_n$.

\item Similarly let $\mathcal{C}DW_n$
be a class of  digraphs (called {\em directed wheel graphs with  spokes}) on $n\geq 4$ vertices such
that
each digraph contains a directed cycle $DC_{n-1}$ (with
vertices labeled as $1,2,\ldots, n-1$) and an additional vertex $n$.
Also for each $i$ either there is a directed arc from $n$ to $i$ or from $i$ to $n$.
Observe  that there are  $2^{n-1}$ such digraphs. Again if $X\in \mathcal{C}DF_n$ and $X$ contains
an arc from  $n-1$ to $n$ and $n$ to $1$  then $X\in \mathcal{C}DC_n$.
\end{itemize}}

\section{Characteristic polynomials and Exponents}\label{sec:one}

\section*{\small{Digraphs from $\mathcal{C}DC_n$}}

We start this section with a well known example. The digraph $DC_n\in
\mathcal{C}DC_n$ is a directed cycle without
chords. It is known that, $DC_n$ is non-derogatory and its minimal
polynomial is $x^n-1=\prod_{d|n}\Phi_d(x)$, where $\Phi_m(x)$ is the $m$-th 
cyclotomic polynomial. It is also known that $A(DC_n)$ (the
adjacency
matrix of $DC_n$) is an irreducible matrix but not
a primitive matrix. From  Theorem 1 of~\cite{hoff}, the complement graph $DC_n^c$, is a
polynomial in $DC_n$. So $A(DC^c_n)$ and
$A(DC_n)$ have the same set of eigenvectors and hence the
characteristic polynomial of $A(DC^c_n)$ is given by
$(x-(n-2))\prod_{d|n,d>1}\Phi_d(-(x+1))$ and
$x(x-(n-2))\prod_{d|n,d>2}\Phi_d(-(x+1))$, when $n$ is odd and even,
respectively.
By definition it is clear that $exp(DC_n^c)=2$ for $n\geq 5$. 
The digraphs
which we consider in the
class $\mathcal{C}DC_n$ and their characteristic polynomials are tabulated in
the following table.  Throughout this paper we suppose $k=\lfloor
\frac{n}{2} \rfloor$, where  $\lfloor x\rfloor$ denotes the largest
integer smaller or equal to $x$.

{\small\begin{tabular}{|l|l|l|l|l|}
\hline
 & $X$ & Directed chords  &  Digraph($DC^{----}_8$) & Characteristic polynomial\\
\hline
1 & $DC_n^{(i,n-i)}$ & $i$ to $n-i$,  &
\begin{tikzpicture}

  [scale=.8,auto=left]

 \node (n1) at (3,7)  {1};
  \node (n2) at (4,7.5)  {2};
    \node (n3) at (5,7.4) {3};
  \node (n4) at (6,7)  {4};
\node (n5) at (6,6)  {5};
\node (n6) at (5,5.6)  {6};
\node (n7) at (4,5.2)  {7};
\node (n8) at (3,6)  {8};
\foreach \from/\to in
{n1/n2,n2/n3,n3/n4,n4/n5,n5/n6,n6/n7,n7/n8,n8/n1,n1/n7,n2/n6,n3/n5}
    \draw [->](\from) -- (\to);
\end{tikzpicture}  & $x^n-\sum_{t=1}^{k-1}x^{n-(2t+1)}-1$ \\
& & $i\in\{1,2,\dots, k\}$ & & \\
 \hline
 2 & $DC_n^{(i,k-i)}$ & $i$ to $k-i$,
&\begin{tikzpicture}
  [scale=.8,auto=left]

\node (n1) at (9,7)  {1};
  \node (n2) at (10,7.5)  {2};
    \node (n3) at (11,7.4) {3};
  \node (n4) at (12,7)  {4};
\node (n5) at (12,6)  {5};
\node (n6) at (11,5.6)  {6};
\node (n7) at (10,5.2)  {7};
\node (n8) at (9,6)  {8};

\foreach \from/\to in
{n1/n2,n2/n3,n3/n4,n4/n5,n5/n6,n6/n7,n7/n8,n8/n1,n1/n3}
    \draw [->](\from) -- (\to);

\end{tikzpicture}   & $x^n-\sum_{i=1}^{\lfloor
\frac{k}{2}\rfloor-1}x^{k-(2i+1)}-1$ \\
& & $i\in\{1,2,\ldots,\lfloor
\frac{k}{2}\rfloor-1\}$ & &\\
 \hline

3 & $DC_n^{(i,k+j+i)}$ & fix $j$, where  $1\leq j\leq k-1$ &
\begin{tikzpicture}
  [scale=.8,auto=left]

  \node (n1) at (9,7)  {1};
  \node (n2) at (10,7.5)  {2};
    \node (n3) at (11,7.4) {3};
  \node (n4) at (12,7)  {4};
\node (n5) at (12,6)  {5};
\node (n6) at (11,5.6)  {6};
\node (n7) at (10,5.2)  {7};
\node (n8) at (9,6)  {8};

  \foreach \from/\to in
{n1/n2,n2/n3,n3/n4,n4/n5,n5/n6,n6/n7,n7/n8,n8/n1,n1/n6,n2/n7,n3/n8}
    \draw [->](\from) -- (\to);
\end{tikzpicture}  & $x^n-(k-j)x^{k+j-1}-1$ \\
 & & $i$ to $k+j+i, \;1\leq i\leq k-j$ & &\\
\hline 4 & $DC_n^{n_1,n_2,\ldots,n_r}$ &choose $r$ such that
&\begin{tikzpicture}
  [scale=.8,auto=left]

  \node (n1) at (9,7)  {1};
  \node (n2) at (10,7.5)  {2};
    \node (n3) at (11,7.4) {3};
  \node (n4) at (12,7)  {4};
\node (n5) at (12,6)  {5};
\node (n6) at (11,5.6)  {6};
\node (n7) at (10,5.2)  {7};
\node (n8) at (9,6)  {8};

\foreach \from/\to in
{n1/n2,n2/n3,n3/n4,n4/n5,n5/n6,n6/n7,n7/n8,n8/n1,n8/n3,n8/n5}
    \draw [->](\from) -- (\to);
\end{tikzpicture}   & $x^n-(\sum_{t=1}^rx^{n_t})-1$ \\
& &$n>n_1>\dots >n_r>0$  & &\\
& &  $n$ to $n_t+1$, $t\in\{1,2,\dots,r\}$, & &\\
 
\hline 
5& $DC_n^{(m)}$ & Fix $m$, where $3\leq m\leq
n-1$  &\begin{tikzpicture}
  [scale=.8,auto=left]
 \node (n1) at (9,7)  {1};
  \node (n2) at (10,7.5)  {2};
    \node (n3) at (11,7.4) {3};
  \node (n4) at (12,7)  {4};
\node (n5) at (12,6)  {5};
\node (n6) at (11,5.6)  {6};
\node (n7) at (10,5.2)  {7};
\node (n8) at (9,6)  {8};
\foreach \from/\to in
{n1/n2,n2/n3,n3/n4,n4/n5,n5/n6,n6/n7,n7/n8,n8/n1,n1/n3,n1/n4,n1/n5,n2/n4,n2/n5,
n3/n5}
    \draw [->](\from) -- (\to);
\end{tikzpicture}   & $x^n-(x+1)^{m-2}$ \\
& &all possible arcs  from $i$ to $j$,  & & \\
& & where $i<j-1$ and $3\leq j\leq m$ & &\\
\hline
\end{tabular}}
\vspace*{5mm}

Before proceeding to the proofs, we use the notation $(1,2,\ldots, n,1)$
to
represent a directed cycle $DC_n$.\\ For example,
$(1,7,8,1),(1,2,6,7,8,1),(1,2,3,5,6,7,8,1)\;and\; (1,2,3,4,5,6,7,8,1)$ are the
only ldsgs of $DC_8^{(i,n-i)}$ hence from Theorem~\ref{thm:coe},
$\Psi_{DC_8^{(i,n-i)}}(x)=x^8-x^5-x^3-x-1$.

\begin{lemma}
Let $n\in\Z^+$ and $k=\lfloor \frac{n}{2}\rfloor$. Then
$\Psi_{DC_n^{(i,n-i)}}(x)=x^n-\sum_{t=1}^{k-1}x^{n-(2t+1)}-1$.
\end{lemma}

\begin{proof}
Let $\Psi_{DC_n^{(i,n-i)}}(x)=x^n+\sum_{i=0}^{n-1}a_ix^{n-i}$. By
definition we have $a_n=-1$ and there are no self loops, parallel
arcs and ldsgs of even length ($<n$) in $DC_n^{(i,n-i)}$. Further
$(1,2,3,\ldots,i, n-i,n-i+1,\ldots,n,1)$ is the only ldsg of length
$2i+1$, for each $i\in\{1,2,\ldots,k-1\}$. Consequently, from
Theorem~\ref{thm:coe}, we have  $a_1=0,\;a_{2i+1}=-1\; and\;a_{2i}=0
\; for\; 1\leq i\leq k-1\;(\;a_{2k}=0\; when\;n=2k+1\;)$. Hence the
result follows.
\end{proof}

\section*{\small{Digraphs from $\mathcal{C}DF_n$}}

{\tiny{\begin{tabular}{|l|l|l|l|l|}
\hline
 & $X$ & Spokes  &  Digraph  & Characteristic polynomial\\
\hline
1 & $ADF_n$ & $1$ to $2i$ and $2i+1$ to $1$
 & \begin{tikzpicture}
  [scale=.5,auto=left]

  \node (n1) at (2,6)  {1};
  \node (n2) at (3,7)  {2};
    \node (n3) at (3,6)  {3};
  \node (n4) at (3,5)  {4};
\node (n5) at (3,4)  {5};

  \foreach \from/\to in {n1/n2,n2/n3,n3/n4,n4/n5,n1/n4,n3/n1,n5/n1}
    \draw [->](\from) -- (\to);

\end{tikzpicture} & $\Psi_{ADF_{2k+1}}(x)=x^{2k+1}-\sum_{i=1}^kix^{2(i-1)}$\\

 &  & where $i=1,2,\ldots,k$ &
&$\;\mbox{and}\;\Psi_{ADF_{2k}}(x)=x\Psi_{ADF_{2k-1}}(x)$\\

\hline
2 & $PDF_n$ & $1$ to $i$, $1\leq i\leq n$ and $n$ to $1$
 & \begin{tikzpicture}
  [scale=.5,auto=left]

  \node (n1) at (2,6)  [circle,draw] {1} edge [in=100,out=60,loop] ();
  \node (n2) at (3,7)  {2};
    \node (n3) at (3,6)  {3};
  \node (n4) at (3,5)  {4};
\node (n5) at (3,4)  {5};

 \foreach \from/\to in {n1/n2,n2/n3,n3/n4,n4/n5,n1/n4,n1/n3,n5/n1,n1/n5}
    \draw [->](\from) -- (\to);

\end{tikzpicture} & $x^n-\sum_{i=1}^nx^{n-i}$\\
 \hline
3 & $kDF_n$ & $1$ to $i$  for  $i\ne n,k$, and $n,k$ to $1$   &
\begin{tikzpicture}
  [scale=.5,auto=left]
   \node (n1) at (14,6)  {1};
  \node (n2) at (15,9)  {2};
    \node (n3) at (15,8.1)  {3};
  \node (n4) at (15,7.2)  {4};
\node (n5) at (15,6.3)  {5};
\node (n6) at (15,5.4)  {6};
\node (n7) at (15,4.5)  {7};

  \foreach \from/\to in
{n1/n2,n2/n3,n3/n4,n4/n5,n5/n6,n6/n7,n3/n1,n1/n4,n1/n5,n1/n6,n7/n1}
    \draw [->](\from) -- (\to);

\end{tikzpicture}  & $(x^n+x^{k-2})-2(\sum_{i=3}^kx^{n-i})-\sum_{i=0}^{n-(k+1)}
x^i$ \\
 \hline
4 & $HDF_n$ & $1$ to $i$ for  $i=1,2,\ldots,k$
&\begin{tikzpicture}
  [scale=.5,auto=left]

 \node (n1) at (14,6)  {1};
  \node (n2) at (15,9)  {2};
    \node (n3) at (15,8.1)  {3};
  \node (n4) at (15,7.2)  {4};
\node (n5) at (15,6.3)  {5};
\node (n6) at (15,5.4)  {6};
\node (n7) at (15,4.5)  {7};

  \foreach \from/\to in
{n1/n2,n2/n3,n3/n4,n4/n5,n5/n6,n6/n7,n1/n3,n4/n1,n5/n1,n6/n1,n7/n1}
    \draw [->](\from) -- (\to);

\end{tikzpicture}  &
$\Psi_{HDF_{2k+1}}(x)=(x^{2k+1}-(k-1)x^{k-1})-\sum_{i=1}^{k-1}i[x^{2k-i-1}+x^{
i-1}
]$ \\
 & & $j$ to $1$ for $j=k+1,\ldots,n$ & &
$\Psi_{HDF_{2k}}(x)=x^{2k}-\sum_{i=1}^{k-1}i[x^{2k-i-2}+x^{i-1}]$\\
\hline
5 & $TDF_n$ & $1$ to $i$  for $i=3j$ or $i=3j+2$,& \begin{tikzpicture}
  [scale=.5,auto=left]

 \node (n1) at (14,6)  {1};
  \node (n2) at (15,9)  {2};
    \node (n3) at (15,8.1)  {3};
  \node (n4) at (15,7.2)  {4};
\node (n5) at (15,6.3)  {5};
\node (n6) at (15,5.4)  {6};
\node (n7) at (15,4.5)  {7};

  \foreach \from/\to in
{n1/n2,n2/n3,n3/n4,n4/n5,n5/n6,n6/n7,n1/n3,n4/n1,n1/n5,n1/n6,n7/n1}
    \draw [->](\from) -- (\to);
\end{tikzpicture} & $(x^{3k}-1)-\sum_{i=1}^{k-1}[x^{3(k-i)-2}+(k-i)x^{3(k-i)-1}
+((k-i)+1)x^{3(k-i)}]$,\\
 &  & $i$ to $1$ for $i=3j+1$, & &
$x^{3k+1}-\sum_{i=1}^{k}((k-i)+1)[x^{3(k-i)}+x^{3(k-i)+1}]$,\\
 & & and $n$ to $1$ & &
$x^{3k+2}-\sum_{i=1}^{k}[x^{3(k-i)}+((k-i)+1)x^{3(k-i)+2}+((k-i)+2)x^{3(k-i)+1}]
$\\
\hline
\end{tabular}}}
\vspace*{4mm}

Now we construct a digraph from $ADF_n$ by adding self loops.
Let $X$ be a digraph constructed from $ADF_n$ by adding $k+1$ and $k$ self loops
at the vertex $1$ when $n=2k+1$ and $n=2k$ respectively. Then its characteristic
polynomial is
$$\Psi_X(x)=
\begin{cases}
x^{2k+1}-\sum_{i=1}^{k+1}ix^{2(i-1)}, & \mbox{ when n=2k+1,}\\
x(x^{2k-1}-\sum_{i=1}^kix^{2(i-1)}), & \mbox{ when n=2k.}
\end{cases}$$
The following observation shows that $\Psi_X(x)$ is irreducible whenever
$n=2k+1$ and is the product of $x$ and an irreducible polynomial when $n=2k$.
Hence $\Psi_X(x)$ is non-derogatory for all $n\geq 3$.

\begin{obs}\label{obs:bra}
 A.T.Brauer~\cite{bra}  proved  that the polynomials of the form
\begin{itemize}
 \item $f_m(x)=x^m-a_1x^{m-1}-a_2x^{m-2}-\dots-a_{m-1}x-a_m$  , where  $m\geq
2$ , $a_1,a_2,\ldots,a_m\in \Z^{+}$ and
$a_1\geq a_2\geq \dots\geq a_m$.
\item  $g_m(x) =
x^{2m+1}\pm(a_1x^{2m} + a_2x^{2m-1}+\dots+ a_{2m+1})$,  where
$a_1> a_3>\dots>a_{2m+1}>0$ and $a_2=a_3=\dots = a_{2m}=0$.
\end{itemize}
are irreducible over $\Q$, the field of rational numbers.
\end{obs}

Now we construct few digraphs from $PDF_n$ by adding self loops or arcs.
\begin{itemize}
\item Let $X_n^{m-1}$ be a digraph constructed from $PDF_n$ by adding $m-1$ self
loops at the vertex $1$ then it is easy to see that
$\Psi_{X_n^{m-1}}(x)=x^n-mx^{n-1}-\sum_{i=0}^{n-2}x^i$. Now from the
Observation~\ref{obs:bra}  $\Psi_{X_n^{m-1}}(x)$ is an irreducible polynomial,
hence $X_n^{m-1}$ is a non-derogatory digraph.
\item  Let $0<n_1<n_2<\dots <n_d<n$ and  $Y_n^{n_1,n_2,\ldots,n_d}$ be an
another digraph constructed from $PDF_n$ by adding an arc from $n_i$ to $1$
where $i=1,2,\ldots,d$ and $m-1(\geq d)$ self loops at the vertex $1$. Then it
is easy to see that
$\Psi_{Y_n^{n_1,n_2,\ldots,n_d}}(x)=x^n-mx^{n-1}-(d+1)(\sum_{i=2}^{n_1}x^{n-i}
)-d(\sum_{i=n_1+1}^{n_2}x^{n-i})-(d-1)(\sum_{
i=n_2+1}^{n_3}x^{n-i})-\dots-2(\sum_{i=n_{d-1}+1}^{n_d}x^{n-i})-\sum_{i=n_d+1
}^nx^{n-i}$. Again from Observation~\ref{obs:bra}
$\Psi_{Y_n^{n_1,n_2,\ldots,n_d}}(x)$ is an irreducible polynomial, hence
$Y_n^{n_1,n_2,\ldots,n_d}$ is a non-derogatory digraph.

\item Now we will construct another class of digraphs $Z_n^j$ from $PDF_n$ by
adding a self loop at the vertex $j\in\{2,3,\ldots,n\}$.  Clearly in this
example $p(L) =1$ for every ldsg of $Z_n^j$ is not true.
\end{itemize}

\begin{lemma}
If $Z_n^j$ be a digraph constructed from $PDF_n$ by adding a self loop at the
vertex $j\in\{2,3,\ldots,n\}$, then
$\Psi_{Z_n^j}(x)=x^n-2x^{n-1}-\sum_{i=0}^{j-3}x^i\;for\; j>2$ and
$\Psi_{Z_n^2}(x)=x^n-2x^{n-1}$.
\end{lemma}
\begin{proof} Let $\Psi_{Z_n^j}(x)=x^n+\sum_{i=1}^na_ix^{n-i}$.
 By definition there are only  two self loops at the vertices $1$ and $j$, so
$a_1=-2$.
$L_1=\{(1,1),(j,j)\}$ and $L_2=\{(1,n,1)\}$ are the only ldsgs of $Z_n^j$ with
exactly two vertices, $p(L_1)=2$ and $p(L_2)=1$ hence $a_2=0$. It is clear that
other coefficients of $\Psi_{Z_n^j}(x)$ depends on the value of $j$.

If $j=2$, then $L_1=\{(2,2),(1,i+1,i+2,\ldots, n,1)\}$ and
$L_2=\{(1,i,i+1,\ldots, n,1)\}$ where $i\in\{2,3,\ldots, n-1\}$ are two ldsgs 
with exactly $n-i+2$ vertices, further $p(L_1)=2, p(L_2)=1$ for every $i$,
consequently $a_t=0\; for\; t=3,4,\ldots,n$. Hence
$\Psi_{Z_n^2}(x)=x^n-2x^{n-1}$.

Now suppose $j>2$.  Now we have to show  $a_{n-i}=-1\;for\;
i=0,1,2,\ldots,j-3$ and $0$, otherwise. It is clear that
$(1,i,i+1,i+2,\ldots,n,1)$ is the only ldsg with $n-i+2$ vertices where
$i=\{2,3,\ldots,j-1\}$, as vertex $j$ needs to be included in the ldsg, hence
$a_{n-i}=-1\; for\; i=0,1,2,\ldots,j-3$. On the other hand if
$i\in\{j,j+1,\ldots,
n-1\}$, then  $L_1=\{(j,j),(1,i+1,i+2,\ldots, n,1)\}$ and
$L_2=\{(1,i,i+1,\ldots, n,1)\}$  are two ldsgs  with exactly $n-i+j$ vertices,
further $p(L_1)=2, p(L_2)=1$ for every $i$, consequently $a_t=0\; for\;
t=3,4,\ldots,n-(j-2)$.
\end{proof}

Now we can construct few more  digraphs having irreducible characteristic
polynomial by adding self loops at the vertex $1$ to the  the digraphs
constructed from $ADF_n$ and $PDF_n$ such that the coefficients of
characteristic polynomials of these digraphs satisfy the  criterion of the
following well known theorem.

\begin{theorem}(Perron's criterion)\label{thm:Perron}
 Let $f(x)=x^n+a_1x^{n-1}+\dots+a_n$ be a polynomial with
integer coefficients. If $|a_1| > 1 + |a_2| + \dots+ |a_n|$, then $f$ is
irreducible.
\end{theorem}

\section*{\small{Digraphs from $\mathcal{C}DW_n$}}
Recall that the class $\mathcal{C}DW_n$ contains a directed cycle
$DC_{n-1}$ (with vertices labeled as $1,2,\ldots, n-1$) and an
additional vertex $n$. Also for each $i$ either there is a directed
arc from $n$ to $i$ or from $i$ to $n$.  For example, if we choose
all the spokes having a unique direction {\it{i.e.,}} all spokes are
from $n$ to $1,2,\ldots,n-1$ or
 from $1,2,\ldots,n-1$ to $n$ and  denote this digraph by  $UDW_n$, then it is easy to see that
$\Psi_{UDW_n}(x)=
 x^n-x$ and  a simple calculation show that for the complement of the
 graph $UDW_n$, we have
$\Psi_{UDW^c_{2k}}(x)=x(\prod_{d|2k-1,d>1}\Phi_{2d}(x+1))(x-(2k-3))\;and\;
\Psi_{UDW^c_{2k+1}}(x)=x^2(\prod_{d|2k,d>2}\Phi_{2d}(x+1))(x-(2k-2))$.
Clearly $UDW_n$ and $UDW_{2k}^c,k\geq 2$ are non-derogatory, whereas
$UDW_{2k+1}^c,k\geq 2$ is derogatory. Also one can check that the
minimal polynomial of $UDW_{2k+1}^c,\;for\;k\geq 2$ is
$x^{-1}(\Psi_{UDW^c_{2k+1}}(x))$. \\ The following table gives few digraphs and
their characteristic polynomials from
the class $\mathcal{C}DW_n$.\\

{\tiny{\begin{tabular}{|l|l|l|l|}
\hline
$X$ & Spokes  &  Digraph  & Characteristic polynomial\\
\hline
 $ADW_n$ & $n$ to $2i-1$ and $2i$ to $n$
 & \begin{tikzpicture}
  [scale=.8,auto=left]

  \node (n1) at (1,4)  {1};
  \node (n2) at (3,4)  {2};
    \node (n5) at (2,3)  {5};
  \node (n3) at (3,2)  {3};
\node (n4) at (1,2)  {4};

  \foreach \from/\to in {n1/n2,n2/n3,n3/n4,n4/n1,n5/n1,n2/n5,n5/n3,n4/n5}
    \draw [->](\from) -- (\to);

\end{tikzpicture} &
$\Psi_{ADW_{2k+1}}=(x^{2k+1}-x)-k\big(\sum_{i=0}^{k-1}x^{2i}\big)$\\

  & where $i=1,2,\ldots,k$ &
&$\Psi_{ADW_{2k}}=(x^{2k}-2x)-\sum_{i=2}^{k-1}ix^{2i-1}-\sum_{j=2}^k(j-1)x^{
2(k-j)}$\\

\hline
 $RADW_{2k+1}$ & $n$ to $2i$, $2i-1$ to $n$ and
 & \begin{tikzpicture}
  [scale=.8,auto=left]
  \node (n1) at (1,4)  {1};
  \node (n2) at (3,4)  {2};
    \node (n5) at (2,3)  {5};
  \node (n3) at (3,2)  {3};
\node (n4) at (1,2)  {4};

  \foreach \from/\to in {n1/n2,n2/n3,n3/n4,n4/n1,n1/n5,n5/n2,n3/n5,n4/n5,n5/n4}
    \draw [->](\from) -- (\to);

\end{tikzpicture} &
$(x^{2k+1}-2x)-k\big(\sum_{i=0}^{k-1}x^{2i}\big)-\sum_{i=1}^{k-1}x^{2i+1}$\\
  &$n-1$ to $n$, $1\leq i\leq k$ &  &\\
 \hline
 $kDW_n$ &$n$ to $i$, $i\ne k$ and $k$ to $n$ & \begin{tikzpicture}
  [scale=.8,auto=left]

 \node (n1) at (13.3,4)  {1};
  \node (n2) at (14.1,4)  {2};
    \node (n3) at (14.9,4)  {3};
  \node (n4) at (15.7,4)  {4};
\node (n8) at (14.5,3.3)  {8};
\node (n5) at (15.7,2.5)  {5};
\node (n6) at (15,2.5)  {6};
\node (n7) at (13.3,2.5)  {7};

  \foreach \from/\to in
{n1/n2,n2/n3,n3/n4,n4/n5,n5/n6,n6/n7,n7/n1,n8/n1,n8/n2,n8/n3,n4/n8,n8/n5,n8/n6,
n8/n7}
    \draw [->](\from) -- (\to);

\end{tikzpicture}& $x^n-x^{n-3}-x^{n-4}-\dots -x^3-x^2-2x-1$.\\
\hline
 $HDW_n$ &$n$ to $i$ for $i=1,2\ldots k$ &\begin{tikzpicture}
  [scale=.8,auto=left]

 \node (n1) at (13.3,4)  {1};
  \node (n2) at (14.1,4)  {2};
    \node (n3) at (14.9,4)  {3};
  \node (n4) at (15.7,4)  {4};
\node (n8) at (14.5,3.3)  {8};
\node (n5) at (15.7,2.5)  {5};
\node (n6) at (15,2.5)  {6};
\node (n7) at (13.3,2.5)  {7};
 \foreach \from/\to in
{n1/n2,n2/n3,n3/n4,n4/n5,n5/n6,n6/n7,n7/n1,n8/n1,n8/n2,n8/n3,n8/n4,n5/n8,n6/n8,
n7/n8}
    \draw [->](\from) -- (\to);
\end{tikzpicture}
 & $(x^{2k+1}-x)-\sum_{i=1}^{k-1}i(x^{i-1}+x^{2k-(i+1)})-kx^{k-1},n=2k+1$\\
 & $j$ to $n$, $j=k+1,..,n-1$& &
$(x^{2k}-x)-\sum_{i=1}^{k-1}i(x^{i-1}+x^{2k-(i+2)}),\;for\;n=2k$\\
\hline

\end{tabular}}}
\paragraph{Distinct eigenvalues:}
We already saw that the characteristic polynomials of few digraphs are
irreducible. It is verified that all the digraphs (except $Z_n^2$)  which are constructed in this
paper have distinct eigenvalues.  We show this for a few of them by using the
following methods.

\begin{itemize}
 \item [Method 1:] If $\gcd(f(x),f'(x))=1$, where $f'(x)$ is the derivative of
$f(x)$, then all the roots of $f(x)$ are distinct. This method can be  applied
for $DC_n^{(i,k+j-i)},\;DC_n^{(m)},\; Z_n^3$.

\item [Method 2:] If $\gcd(f(x),f'(x))=1$ in $\Z_2[x]$, where $\Z_2$ is a finite
field with $2$ elements, then all the roots of $f(x)$ are distinct. The digraphs
 $DC_{2k+1}^{(i,n-i)},\; ADF_{2k+1}, DC_n^{(i,k-i)}$(for $n$ and $k$  odd)
can be shown to  have distinct eigenvalues by this method.
\item [Method 3:] By complete factorization. For example,\\
$\Psi_{ADW_{2k+1}}(x)= x(x^{2k}-1)-k(1+x^2+x^4+\dots+ x^{2(k-1)})=
(x^3-x-k)\prod_{d|2n,d>2} \Phi_d(x)\\
\;similarly \;\Psi_{RADW_{2k+1}}(x)=(x^3-2x-k)\prod_{d|2n,d>2}\Phi_d(x)$.
\end{itemize}

\paragraph{Exponents:}
We now compute the exponents of some of these digraphs for $n\geq 10$.
\begin{lemma}
Let $ADF_n$ be a  directed fan with alternating spokes of order $n=2k+1$, where
$k>2$. Then $exp(ADF_n)=9$
\end{lemma}

\begin{proof} First observe that there is no walk of length of $8$ from vertex
$n-1$ to 3.
 Hence if we show that there is a walk of length $9$ between any two vertices of
$DF_n$, then the result follows from the Theorem~\ref{thm:walk}.
Note that  all additions here are done  under modulo $n$, whenever
sum exceeds n. Also observe that $ADF_n$ has at least $7$ vertices
as $k>2$. A walk of length $9$ are given as follows.
\begin{itemize}
\item 1 to 1 is $(1,2,3,1,2,3,1,2,3,1)$.
\item 1 to i, where $i$ is even is $(1,2,3,1,2,3,4,5,1,i)$.
\item 1 to i, where $i$ is odd is $(1,2,3,4,5,6,7,1,i-1,i)$.
\item i to 1, where $i$ is even is $(i,i\pm 1,1,2,3,4,5,6,7,1)$.
\item i to 1, where $i$ is odd is  $(i,1,2,3,4,5,1,i-1,i,1)$.
\item i to j, where $i,j$ are even is $(i,i\pm 1,1,i,i\pm 1,1,i,i+1,1,j).$
\item i to j, where $i,j$ are odd is $(i,1,2,3,1,2,3,1,j-1,j).$
\item i to j, where $i$ is even and $j$ is odd is $(i,i\pm 1,1,2,3,4,5,1,j-1,j)$.
\item i to j, where $i$ is odd and $j$ is even is $(i,1,2,3,4,5,6,7,1,j)$.
\end{itemize}
\end{proof}

By definition $ADF_n$ is reducible, whenever $n$ is even, as the
last row of the adjacency matrix of $ADF_n$ is the zero row, whereas
$ADF_3$ is a directed cycle. This is an example of a digraph with
non primitive irreducible adjacency matrix and  $exp(ADF_5)=12$.

The following table gives the exponents of some of the digraphs.\\

{\tiny{\begin{tabular}{|l|l|l|l|l|l|l|l|l|l|}
 \hline
$X$ & $ADF_{2k+1}$ & $PDF_n$ & $kDF_{2k}$ & $kDF_{2k+1}$ & $HDF_n$ &
$ADW_{2k+1}$ & $ADW_{2k}$ & $kDW_{2k}$ & $kDW_{2k+1}$  \\
\hline
$exp(X)$ & 9 for $n\ge 6$ & n  &  k+4 & k+5 & n+1 & 6 & 7 & 2k+3 & 2k+4 \\
\hline
no walk of &  &   &   &  &  &  &  &  &  \\
length  & n-1 to 3 & n-1 to 2   & k+1 to 2  & k+1 to 2  & 2 to n &n-2 to 2  &n-3
to 2  &k+1 to k+2  & k+1 to k+2 \\
$exp(X)-1$ from  &  &   &   &  &  &  &  &  &  \\
\hline
\end{tabular}}}

\end{document}